\documentclass[11pt]{article}
\usepackage{exscale,relsize}
\usepackage{amsmath}
\usepackage{amsfonts}
\usepackage{amssymb}
\usepackage{calc}
\usepackage{theorem}
\usepackage{pifont}      
\newcommand{\im}{\mathrm{Im}}
\newcommand{\dom}{\mathrm{dom}}

\newcommand{\RX}{\ensuremath{\,(-\infty,+\infty]}}

\newcommand{\RR}{\ensuremath{\mathbb R}}
\newcommand{\cA}{\mathcal{A}}

\newcommand{\cF}{\mathcal{F}}

\newcommand{\D}{\Delta}

\newtheorem{theorem}{Theorem}[section]
\newtheorem{definition}[theorem]{Definition}

\newtheorem{corollary}[theorem]{Corollary}
\newtheorem{example}[theorem]{Example}
\theoremstyle{plain}{\theorembodyfont{\rmfamily}}
\newtheorem{proposition}[theorem]{Proposition}

\theoremstyle{plain}{\theorembodyfont{\rmfamily}

\theoremstyle{plain}{\theorembodyfont{\rmfamily}
\newtheorem{remark}[theorem]{Remark}

\begin{document}
\title{Abstract Convex Optimal Antiderivatives }

\author{
Sedi Bartz\thanks{Department of Mathematics,
The Technion -- Israel Institute of Technology,
32000 Haifa, Israel.
E-mail: \texttt{bartz@techunix.technion.ac.il}.}
\ \ and\ \
Simeon Reich\thanks{Department of Mathematics,
The Technion -- Israel Institute of Technology,
32000 Haifa, Israel.
E-mail: \texttt{sreich@techunix.technion.ac.il}.}
}

\date{December 03, 2011}

\maketitle


\begin{abstract} \noindent
Having studied families of antiderivatives and their envelopes in the setting of classical convex analysis, we now extend and apply these notions and results in settings of abstract convex analysis. Given partial data regarding a $c$-subdifferential, we consider the set of all $c$-convex $c$-antiderivatives that comply with the given data. Under a certain assumption, this set is not empty and contains both its lower and upper envelopes. We represent these optimal antiderivatives by explicit formulae. Some well known functions are, in fact, optimal $c$-convex $c$-antiderivatives. In one application, we point out a natural minimality property of the Fitzpatrick function of a $c$-monotone mapping, namely that it is a minimal antiderivative. In another application, in metric spaces, a constrained Lipschitz extension problem fits naturally the convexity notions we discuss here. It turns out that the optimal Lipschitz extensions are precisely the optimal antiderivatives. This approach yields explicit formulae for these extensions, the most particular case of which recovers the well known extensions due to McShane and Whitney.
\end{abstract}



\section{Introduction}

\emph{Abstract convex analysis} (or \emph{generalized convexity}) has attracted more and more attention in recent years. Now this topic is studied continually, both theoretically and from the point of view of applications. It has turned out to be a useful tool for solving optimization problems which are not of convex nature in the classical sense. For recent examples see \cite{burrub, pen, rol, rub, sin} and references therein. For instance, abstract convex analysis has been applied to duality theory for global optimization problems. In another particular application, $c$-convexity and $c$-cyclic monotonicity are key features in grasping, as well as in recent proofs of, the Kantorovich duality theorem in optimal transport. Optimal transport plans are $c$-cyclically monotone and admit $c$-convex $c$-antiderivatives. A detailed presentation of this topic and its history can be found in \cite{vil}. Such applications of cyclic monotonicity to optimal transport can be traced back as early as to \cite{Bre1, Bre2}. Our aim here is to introduce and study families of $c$-convex $c$-antiderivatives and in particular, optimal $c$-convex $c$-antiderivatives. This extends our recent discussion in \cite{br}, which took place in the context of classical convex analysis. The general framework of abstract convex analysis allows us to apply optimal antiderivatives in situations which are not of convex nature in the classical sense. These applications shed some new light and extend previous studies in the theory of representation of monotone mappings by convex functions and regarding the problem of extending Lipschitz functions, as we describe in more detail below.

The following well known definitions are the foundations of our discussion. Unless otherwise specified, throughout the paper $X$ and $Y$ are arbitrary  sets and $c:X\times Y\to\RR$ is an arbitrary function. We say that a function $f:X\to\RX$ is proper if $\dom (f):=\{x\in X\ |\ f(x)<\infty\}$ is not empty.

\begin{definition}[c-transform]
Let $X$ and $Y$ be nonempty sets and let $c:X\times Y\to\mathbb{R}$ be a
function. Given
a function $f:X\to[-\infty,\infty]$, its $c$-transform $f^c:Y\to[-\infty,\infty]$ is defined by
\begin{equation}
f^c(y):=\sup_{x\in X}\ c(x,y)-f(x),\ \ \ y\in Y.
\end{equation}
Similarly, the $c$-transform of a function $g:Y\to[-\infty,\infty]$ is the function $g^c:X\to[-\infty,\infty]$ defined by
\begin{equation}
g^c(x):=\sup_{y\in Y}\ c(x,y)-g(y),\ \ \ x\in X.
\end{equation}
\end{definition}

\begin{definition}[c-convexity]\label{cconvexity}
A proper function $f:X\to\RX$ is said to be $c$-convex if there exists a (necessarily proper) function $g:Y\to\RX$ such that $f=g^c$. The set of all $c$-convex functions defined on $X$ is denoted by $\Gamma_c(X)$. The function $(f^c)^c$ is the $c$-convexification of $f$ and is denoted by $f^{cc}$.
\end{definition}

The $c$-transform of a function is also known as its \emph{c-conjugate} function. This generalization of Fenchel's conjugate function from classical convex analysis was introduced and studied by Moreau in \cite{mor}. In other, more general, discussions, one encounters convexity of a function with respect to a family of elementary (sometimes referred to as \emph{affine}) functions defined on $X$. That is, a function is said to be convex with respect to a set of functions $A$ if it is the upper envelope of a subset of $A$. In our discussion, $A=\{c(\cdot,y)+r\ |\ y\in Y,\ r\in\mathbb{R}\}$. Sometimes, the function $c$ is allowed to take the values $\pm\infty$. Many other distinctions within our definitions and beyond them, and the theories to which they lead, were considered by different authors; see, for example, \cite{rub, sin}. In our discussion we focus our attention on the settings in the above definitions. When referring to the theoretical study of $c$-convexity, the function $c$ is called a \emph{coupling function} between $X$ and $Y$, while in the particular application to the study of optimal transport, as in \cite{vil}, it is said to be the \emph{cost function}.

Clearly,
\begin{equation}
c(x,y)\leq f(x)+f^c(y)\ \ \ \mathrm{for\ all}\ x\in X\ \mathrm{and}\ y\in Y.
\end{equation}
The case of equality is captured in the following definition of the $c$-subdifferential. We will denote the graph of a multivalued mapping $M:X\rightrightarrows Y$ by $G(M):=\{(x,y)|\ y\in M(x)\}$. The mapping $M$ is called proper if $\dom(M):=\{ x\in X |\ M(x)\neq\emptyset\}$ is not empty.

\begin{definition}[c-subdifferential and c-antiderivative]\label{ctrans}
Let $f:X\to\RX$ be a proper function. The $c$-subdifferential of $f$ is the mapping
$\partial_c f:X\rightrightarrows Y$ defined by
\begin{align}
\partial_c f(x):=\ &\big\{y\in Y\ |\ f(x)+c(x',y)\leq f(x')+c(x,y)\ \ \forall x'\in X\big\}\\
\nonumber\\
=\ &\big\{y\in Y\ |\ f(x)+f^c(y)=c(x,y)\big\}.
\end{align}
When $\partial_c f(x)\neq\emptyset$, we say that $f$ is $c$-subdifferentiable at $x$. When $M:X\rightrightarrows Y$ and $G(M)\subset G(\partial_c f)$, we say that $f$ is a $c$-antiderivative of $M$.
\end{definition}

For a proper function $f:X\to\RX$, it is clear that $\dom(\partial_c f)\subset\dom (f)$, that is, $f$ is finite where it is $c$-subdifferentiable. We will  continue our discussion of elementary $c$-convex analysis in Section 2.

Our main interest here is in solving the following problem and then applying our solution in particular situations. Suppose that $f:X\to\RX$ is a $c$-antiderivative of the mapping $M:X\rightrightarrows Y$ and suppose that $S$ is a nonempty subset of $\dom (M)$. We are interested in the $c$-convex solutions $h:X\to\RX$ of the problem
\begin{equation}\label{prob}
G(M)\subset G(\partial_c h)\ \ \ \ \ \mathrm{such\ that}\ \ \ \ \ h|_S=f|_S.
\end{equation}
One may consider this as a constrained $c$-convex extension problem. We are interested in extending the function $f|_S$ to a $c$-convex function defined on $X$ while keeping it a $c$-antiderivative of $M$. In Section 3 we address this problem and show that the family of solutions, denoted by $\cA_{[c,f|_S,M]}$, is not empty and contains its upper envelope $\gamma_{[c,f|_S,M]}$. This envelope is then the maximal $c$-convex $c$-antiderivative of $M$ that agrees with $f$ on $S$. We apply duality relations to conclude that there also is a minimal $c$-convex $c$-antiderivative $\alpha_{[c,f|_S,M]}\in\cA_{[c,f|_S,M]}$. In the special case where $S=\dom (M)$, we deduce explicit formulae for $\alpha_{[c,f|_S,M]}$ and $\gamma_{[c,f|_S,M]}$.

It is well known that Rockafellar's famous characterization \cite{roc} of cyclically monotone mappings from classical convex analysis extends to $c$-cyclically monotone mappings with respect to $c$-convexity. Moreover, Rockafellar's antiderivative plays the same r\^{o}le as in the classical case. It is an explicitly represented $c$-convex $c$-antiderivative of a given $c$-cyclically monotone mapping. In Section 4 we recall these facts and then apply Rockafellar's antiderivative in order to explicitly construct the minimal antiderivative in the general case of \eqref{prob}. Besides this advantage, this approach is an alternative way to the one in Section 3. It enables us to solve problem \eqref{prob} without appealing to duality results.

Let $(X,d)$ be a metric space. Recall that a function $f:X\to\mathbb{R}$ is called $K$-Lipschitz with $K\geq 0$ if
$$
|f(x)-f(y)|\leq Kd(x,y)\ \ \ \ \ \ \ \ \ \ \mathrm{for\ all}\ x,y\in X.
$$
Let $S$ be a nonempty subset of $X$ and let $f:S\to\mathbb{R}$ be $K$-Lipschitz. Then the well known extension theorem of McShane \cite{mcs} and Whitney \cite{whi} asserts that $f$ extends to a $K$-Lipschitz function which is defined on all of $X$. In particular, the minimal and maximal $K$-Lipschitz extensions of $f$ are given explicitly by
\begin{equation}\label{McSWhi}
\alpha(x)=\sup_{s\in S}\ [f(s)-Kd(x,s)]\ \ \ \ \ \ \ \ \ \mathrm{and}\ \ \ \ \ \ \ \ \ \gamma(x)=\inf_{s\in S}\ [f(s)+Kd(x,s)],
\end{equation}
respectively. In Section 5 we address the following constrained Lipschitz extension problem: Suppose that $f:A\subset X\to\mathbb{R}$ is a $K$-Lipschitz function and let $G$ be a subset of $A\times A$ such that
\begin{equation}\label{cond}
f(y)-f(x)=Kd(x,y)\ \ \ \ \ \ \mathrm{for\ all}\ \ (x,y)\in G.
\end{equation}
We set $G=G(M)$ where $M:A\rightrightarrows A$ is the mapping the graph of which is the set $G$. Given a nonempty subset $S$ of $\dom (M)$, the problem is to extend $f|_S$ to a $K$-Lipschitz function which is defined on all of $X$ and still satisfies (\ref{cond}). In particular, we look for the minimal and maximal extensions. We accomplish this and provide formulae for  the minimal and maximal $K$-Lipschitz extensions that satisfy (\ref{cond}). Apart from particular cases, formulae \eqref{McSWhi} are not necessarily solutions to this generalized problem. If one assumes the McShane and Whitney extension result, then the existence of optimal Lipschitz extensions in our discussion is a rather simple consequence. However, approaching this problem through $-d$-convexity, we do not need to rely on the McShane and Whitney theorem. Furthermore, this approach leads to the representation of the optimal extensions with explicit formulae of which formulae (\ref{McSWhi}) are a particular case.  When $c=-d$, then $-d$-convex functions are precisely the $1$-Lipschitz functions. In this case, $f$ is a $-d$-antiderivative of $M$ and the problem fits naturally in the context of optimal antiderivatives. We note that this discussion is never empty since the identity mapping can play the r\^{o}le of $M$ with respect to any function $f$. This particular case recovers the McShane and Whitney extensions from our discussion. Several approaches to the theory of Lipschitz functions and to the problem of extending Lipschitz functions with tools from classical convex analysis and abstract convex analysis have been proposed in the last three decades. For instance, in Section 5 we generalize previous dicussions and results in this direction which appear in \cite{evemaa, hir, leg, rol, sin}.

In the framework of classical convex analysis and monotonicity, the associated Fitzpatrick function of a monotone mapping was first defined and studied in \cite{fit}. Fitzpatrick also considered an associated family of functions with a similar underlying connection to the given monotone mapping, now known as the Fitzpatrick family. It was the beginning of what has become an intense study, often referred to as the ``representation of monotone operators by convex functions''. Apart from their theoretical interest, these functions found application in new accessible convex analytic proofs of sum and range theorems for monotone mappings and of other results in monotone operator theory. A well known property of the Fitzpatrick function, which was originally observed by Fitzpatrick in \cite{fit}, is that when the underlying mapping is maximal monotone, then the Fitzpatrick function is the minimal function in the Fitzpatrick family. In \cite{br} we studied how a monotone mapping gives rise to an associated family of convex antiderivatives and we identified the Fitzpatrick function as the minimal convex antiderivative. This extended the notion of minimality of the Fitzpatrick function to also hold for monotone mappings which are not necessarily maximal monotone. In Section 6 we generalize most of the discussion in \cite{br} and show that it continues to hold in the setting of $c$-convexity and $c$-monotone mappings.

\section{$c$-Convex Analysis Preliminaries}

We proceed with our discussion of abstract convexity by recalling some basic properties of $c$-convex functions and $c$-antiderivatives, which we utilize later and which are easily verified using the definitions:\\
\\
$\bullet$ For any constant $C$, if $f$ is a $c$-antiderivative of $M$, then $f+C$
is also a $c$-antiderivative of $M$.\\
\\
$\bullet$ For any constant $C$, If $f$ is $c$-convex, say $f=g^c$, then
$f+C$ is also $c$-convex since $f+C=(g-C)^c$.\\
\\
$\bullet$ The $c$-transform reverses order. That is, $h\leq f\ \Rightarrow\ f^c\leq h^c$.

Recall that the indicator function of a set $S\subset X$ is the function $\iota_{S}:X\to\RX$, defined by
$$
\iota_{S}(x)=\Big\{\begin{array}{c}
                     0\ \ \ x\in S \\
                     \infty\ \ x\notin S.
                   \end{array}
$$
\\
$\bullet$ For every $y\in Y$, the function $c(\cdot,y):X\to\RX$ is $c$-convex
since $c(\cdot,y)=\iota_{\{y\}}^c$.\\

\begin{proposition}\label{subinclusion}
Let $f,g:X\to\RX$ satisfy $f\leq g$. If at the point $x\in X$ the equality $f(x)=g(x)$ holds, then $\partial_c f(x)\subset\partial_c g(x)$.
\end{proposition}

\begin{proof}
Suppose that $y\in\partial_c f(x)$. Then for every $x'\in X$ we have
$$
g(x)+c(x',y)=f(x)+c(x',y)\leq f(x')+c(x,y)\leq g(x')+c(x,y),
$$
which implies that $y\in\partial_c g(x)$.
\end{proof}

It is a crucial and elementary fact in our discussion that Definitions \ref{cconvexity} and \ref{ctrans} are preserved by upper envelopes. Unless otherwise specified, throughout the paper the families we consider are indexed by an arbitrary index set. 

\begin{proposition}[convexity of the upper envelope of $c$-convex functions]\label{upenvconv}
If $\{f_s\}$ is a family of functions such that $f_s:X\to\RX$ is $c$-convex for every $s$, then the upper envelope, $f=\sup_s f_s$, is also $c$-convex, when it is proper.
\end{proposition}

\begin{proof}
Suppose that for every $s$, the function $f_s$ is $c$-convex, say $f_s=g^c_s$. Then the function $g:=\inf_s g_s$ does not attain the value $-\infty$, that is $g:Y\to\RX$. Indeed, suppose that for some $y_0\in Y$ we have $\inf_s g_s(y_0)=-\infty$. Let $\{s_n\}$ be a sequence such that $g_{s_n}(y_0)\to-\infty$. Then for every $x\in X$, we obtain
$$
f(x)=\sup_s f_s(x)\geq\sup_n f_{s_n}(x)=\sup_n\sup_y[c(x,y)-g_{s_n}(y)]\geq\lim_{n\to\infty}[c(x,y_0)-g_{s_n}(y_0)]=\infty,
$$
which implies that $f$ cannot be a proper function in this case. A straightforward verification now yields $f=g^c$.
\end{proof}

\begin{proposition}[upper envelope of $c$-antiderivatives of a common mapping]\label{upenvantider}
If $\{f_s\}$ is a family of functions such that $f_s:X\to\RX$ is a $c$-antiderivative of the mapping $M:X\rightrightarrows Y$ for every $s$, then the upper envelope, $f=\sup_s f_s$, is also a $c$-antiderivative of $M$, when it is proper.
\end{proposition}

\begin{proof}
Let $(x,y)\in M$ and $x'\in X$. Since $f_s$ is a $c$-antiderivative of $M$ for every $s$, we have
\begin{equation}
f_s(x)+c(x',y)-c(x,y)\leq f_s(x')\leq f(x')
\end{equation}
for every $s$. It follows that $f(x)+c(x',y)-c(x,y)\leq f(x')$, and
consequently, that $f$ is a $c$-antiderivative of $M$.
\end{proof}

A major part of the theory of $c$-convex functions is made possible by the following well known fundamental property of the $c$-transform \cite{mor}:

\begin{proposition}[c-convexification criterion for c-convexity]\label{ccc}
If $f:X\to[-\infty,\infty]$ is any function, then
$$
f^{ccc}=f^c.
$$
Consequently, a proper function $f$ is $c$-convex if and only if $f^{cc}=f$.
\end{proposition}

\begin{proof}
By definition, for any $y\in Y$ we have
\begin{equation}
f^{ccc}(y)=\sup_x\inf_{y'}\sup_{x'}\ [c(x,y)-c(x,y')+c(x',y')-f(x')].
\end{equation}
Letting $y'=y$, we see that $f^{ccc}(y)\leq f^c(y)$, while letting $x'=x$ we obtain that $f^{ccc}(y)\geq f^c(y)$. Now, if $f$ is $c$-convex, then $f=g^c$ for some function $g:Y\to\RX$ and so $f^{cc}=g^{ccc}=g^c=f$. Conversely, if $f=f^{cc}$, then $f$ is $c$-convex by definition because it is the $c$-transform of the function $f^c:Y\to\RX$.
\end{proof}

We can now justify the use of the notion of $c$-convexification as follows:

\begin{corollary}[c-convexification property]
For any proper function $f:X\to\RX$, $f^{cc}$ is the largest $c$-convex function majorized by $f$. That is, if $h:X\to\RX$ is $c$-convex and $h\leq f$, then $h\leq f^{cc}\leq f$.
\end{corollary}

\begin{proof}
By definition, for any $x\in X$ we have
\begin{equation}
f^{cc}(x)=\sup_{y}\inf_{x'}\ [c(x,y)-c(x',y)+f(x')].
\end{equation}
Letting $x'=x$, we conclude that $f^{cc}(x)\leq f(x)$. Now, if $h:X\to\RX$ is a $c$-convex function such that $h\leq f$, then applying the fact that the $c$-transform reverses order and the above criterion for $c$-convexity, we get $h=h^{cc}\leq f^{cc}\leq f$.
\end{proof}

Furthermore, we can now relate the $c$-sundifferential of a function with the $c$-subdifferential of its $c$-transform. Recall that the mapping $M^{-1}:Y\rightrightarrows X$ is defined by $x\in M^{-1}(y)\Leftrightarrow y\in M(x)$.

\begin{corollary}[$c$-subdifferential of the $c$-transform]\label{duality}
If the function $f:X\to\RX$ is a $c$-antiderivative of the mapping $M:X\rightrightarrows Y$, then $f^c$ is a $c$-antiderivative of $M^{-1}$. Moreover, if $f$ is $c$-subdifferentiable at $x$, then $f^{cc}(x)=f(x)$. If, in addition, $f$ is $c$-convex, then 
\begin{equation}
\partial_c f^c=(\partial_c f)^{-1}.
\end{equation}
\end{corollary}

\begin{proof}
Let $f$ be a $c$-antiderivative of $M$. If $f$ is $c$-convex, then $f^{cc}=f$. It follows that the set of points $(x,y)$ where the equality  $c(x,y)=f^{cc}(x)+f^c(y)$ holds is precisely the set of points where the equality $c(x,y)=f(x)+f^c(y)$ holds. This implies that $\partial_c f^c=(\partial_c f)^{-1}$. In the general case, since $f^{cc}\leq f$, then for every $(x,y)$,
\begin{equation}\label{ctranssub}
c(x,y)\leq f^{cc}(x)+f^c(y)\leq f(x)+f^c(y).
\end{equation}
It now follows that the set of points $(x,y)$ where the equality $c(x,y)=f(x)+f^c(y)$ holds is contained in the set of points where the equality $c(x,y)=f^{cc}(x)+f^c(y)$ holds, that is, $(x,y)\in G(\partial_c f) \Rightarrow (y,x)\in G(\partial_c f^c)$, which means that $G((\partial_c f)^{-1})\subset G(\partial_c f^c)$. Finally, if $f$ is $c$-subdifferentiable at $x$, say $y\in\partial_c f(x)$, then both the inequalities in \eqref{ctranssub} become equalities and we obtain $f^{cc}(x)=c(x,y)-f^c(y)=f(x)$.
\end{proof}
\\

At this point we describe the manner in which classical convex analysis is captured by $c$-convex analysis. Let $X$ be a locally convex separated topological vector space, let $X^*$ be its dual and set $\langle x^*,x\rangle=x^*(x)$ for $x\in X$ and $x^*\in X^*$. Given a function $f:X\to\RX$, its Fenchel conjugate is the function $f^*:X^*\to\RX$ defined by
$$
f^*(x^*)=\sup_{x\in X}\ \langle x^*,x\rangle-f(x).
$$
We see that when we let $Y=X^*$ and $c(x,y)=\langle y,x\rangle$ for $x\in X$ and $y\in Y$, then $f^c=f^*$ for a function $f:X\to\RX$. However, we have $f^{cc}=f^{**}|_{X}$, when we identify $X$ with its canonical embedding in $X^{**}$. In this setting, if $f$ is proper, then it is well known that $f$ is convex and lower semicontinuous if and only if $f=f^{cc}$. The characterization of $c$-convex functions obtained in Proposition \ref{ccc}, namely, that the proper function $f$ satisfies $f=f^{cc}$ if and only if $f$ is $c$-convex, is not a direct generalization of the above fact from classical convex analysis. It is, of course, true that in the classical setting, the proper function $f$ is $c$-convex if and only if $f=f^{cc}$. However, the fact that $f$ is $c$-convex if and only if $f$ is lower semicontinuous and convex is not a consequence of Proposition \ref{ccc}. We are thus led to consider the following observations. Recall that the epigraph of a function $f:X\to\RX$ is the set $\{(x,t)\in X\times\mathbb{R}\ |\ f(x)\leq t\}$. Perhaps, since a proper function is $c$-convex if and only if its epigraph is the non-empty intersection of a family of epigraphs of $c$-convex functions, one should say that a ``$c$-convex'' function is, in fact, ``outerly $c$-convex''. In the classical definition of a convex function, convexity can be referred to as ``inner convexity'', that is, through line segments inside the epigraph. For a proper and lower semicontinuous function, what unifies ``inner convexity'' and ``outer convexity'' in the classical case, that is, the reason that a proper function is ``innerly $c$-convex'' (convex and lower semicontinuous) if and only if it is ``outerly $c$-convex'' (c-convex), is the Hahn-Banach separation theorem.

In Sections 5 and 6 we study families of functions which turn out to be families of $c$-convex $c$-antiderivatives. The families in Section 5 are convex. The fact that the families in Section 6 are closed under nontrivial convex combinations is well known in the classical case. Before ending this section we draw the reader's attention to the $c$-convexity and $c$-subdifferentiability properties of nontirivial convex combinations of functions. We refer to these properties in the next section where we define families of $c$-convex $c$-antiderivatives.

\begin{proposition}\label{convcomb}
Suppose that $g$ and $h$ are $c$-antiderivatives of the mapping $M:X\rightrightarrows Y$ and let $\lambda\in (0,1)$. Then $\lambda g+(1-\lambda)h$ is a $c$-antiderivative of $M$.
\end{proposition}

\begin{proof}
Let $(x,y)\in G(M)$. Then $y\in \partial_c g(x)$ and $y\in \partial_c h(x)$. For every $x'\in X$, summing up the inequalities
$$
\lambda g(x)+\lambda c(x',y)\leq\ \lambda g(x')+\lambda c(x,y)
$$
and
$$
(1-\lambda) h(x)+(1-\lambda) c(x',y)\leq\ (1-\lambda) h(x')+(1-\lambda) c(x,y),
$$
we get
$$
(\lambda g+(1-\lambda) h)(x)+ c(x',y)\leq\ (\lambda g+(1-\lambda) h)(x')+ c(x,y),
$$
which implies that $y\in \partial_c (\lambda g+(1-\lambda)h)(x)$ and consequently that $G(M)\subset G(\partial_c (\lambda g+(1-\lambda)h))$.

\end{proof}

Convex combinations of $c$-convex functions may fail, in general, to be $c$-convex. Examples are easily constructed. In particular, we end this section with the following example where it is also easy to identify all $c$-convex functions on $X$ and $Y$.

\begin{example}\label{exampleA}
Let $X=\mathbb{R}$ and let $Y$ be a set of two points, say $Y=\{a,b\}$. Let $c:X\times Y\to\mathbb{R}$ be defined by
$$
c(x,y)=\Big\{\begin{array}{c}
                \ x\ \ \ \ \ \ \ \ y=a\\
                -x\ \ \ \ \ \ \ \ y=b.
              \end{array}
$$
Then the functions in $\Gamma_c(X)$, that is, the $c$-convex functions on $X$,  are the functions of the following types:
\begin{equation}\label{example}
1)\ x-p,\ \ \ \ \ \ 2)\ -x-q,\ \ \ \ \ \ 3)\ |x-r|-s,\ \ \ \ \ \ \ \ p,q,r,s\in\mathbb{R}.
\end{equation}
Consequently, $\Gamma_c(X)$ is not closed under nontrivial convex combinations. In particular, the functions $x$ and $|x|$ are $c$-convex $c$-antiderivatives of the mapping $M:X\rightrightarrows Y$, defined by $G(M):=\{(x,a)|\ 0\leq x\}$. Their convex combination
$$
\frac{x}{2}+\frac{|x|}{2}=\Big\{\begin{array}{c}
                                       0\ \ \ \ \ \ x\leq 0 \\
                                       x\ \ \ \ \ \ 0\leq x
                                     \end{array}
$$
is a $c$-antiderivative of $M$ which fails to be $c$-convex. On the other hand, all proper functions on $Y$ are $c$-convex.
\end{example}

\begin{proof}
We prove \eqref{example}. Since there are 3 types of proper functions $g:Y\to\RX$,  these functions have 3 types of $c$-transforms, that is, 3 types of $c$-convex functions on X:
\begin{align}
1)\ \ \ &g(a)\in\mathbb{R}\ \ \mathrm{and}\ \ g(b)=\infty\ \ \Rightarrow\ \ \ \ g^c(x)=x-g(a);\nonumber\\
\nonumber\\
2)\ \ \ &g(a)=\infty\ \ \mathrm{and}\ \ g(b)\in\mathbb{R}\ \ \Rightarrow\ \ \ \ g^c(x)=-x-g(b);\nonumber\\
\nonumber\\
3)\ \ \ &g(a)\in\mathbb{R}\ \ \mathrm{and}\ \ g(b)\in\mathbb{R}\ \ \ \Rightarrow\ \ \ \  g^c(x)=\big|x-(g(a)-g(b))/2\big|-(g(a)+g(b))/2.\nonumber
\end{align}
In cases 1 and 2 we see that all real values for $p$ and $q$ in \eqref{example} are possible. In the third case, since the transformation $(g(a),g(b))\mapsto (1/2)(g(a)-g(b),g(a)+g(b))$ is onto $\mathbb{R}^2$, all real values for $r$ and $s$ in \eqref{example} are possible. In all three cases, $c$-transforming the function $g^c$, we get $g=g^{cc}$.
\end{proof}

\section{The Family $\cA_{[c,f|_s,M]}$ of Antiderivatives and Duality}

Our main purpose is to extend the notions and results from classical convex analysis which we have presented in \cite{br} to the setting of $c$-convexity. To this end, we first consider families of $c$-convex $c$-antiderivatives as follows.

\begin{definition}
Given a mapping $M:X\rightrightarrows Y$, a $c$-antiderivative $f$ of $M$ and a subset $S$ of $\dom (M)$, we denote the set of all $c$-convex functions $h:X\to\RX$ which satisfy
\begin{equation}
G(M)\subset G(\partial_c h)\ \ \ \ and \ \ \ h|_S=f|_S
\end{equation}
by $\cA_{[c,f|_S,M]}$.
\end{definition}

In the above setting, since $M(S)$ is a subset of $\dom (M^{-1})$, Corollary \ref{duality} makes it is also possible to consider the $c$-dual problem: $f^c$ is a $c$-antiderivative of $M^{-1}$. We therefore denote the set of $c$-convex solutions $h:Y\to\RX$ of the problem
\begin{equation}
G(M^{-1})\subset G(\partial_c h)\ \ \ \ and \ \ \ h|_{M(S)}=f^c|_{M(S)}
\end{equation}
by $\cA_{[c,f^c|_{M(S)},M^{-1}]}$.

In the following result we use the $c$-convexification in order to establish the nonemptiness of $\cA_{[c,f|_S,M]}$. Then the existence of a maximal element follows. We collect duality relations between $\cA_{[c,f|_S,M]}$ and $\cA_{[c,f^c|_{M(S)},M^{-1}]}$ which, in particular, imply the existence of a minimal member. By $c$-transforming the maximal element of one family we arrive at the minimal element of the dual family and vice versa.

\begin{theorem}\label{main}
Suppose that $f:X\to\RX$ is a $c$-antiderivative of the mapping $M:X\rightrightarrows Y$.  Suppose further that $\emptyset\neq S\subset\mathrm{dom} (M)$. Then $\cA_{[c,f|_S,M]}$ is nonempty and contains both its upper envelope, that is, the function $\gamma_{[c,f|_S,M]}:X\to\RX$ defined by
$$
\gamma_{[c,f|_S,M]}(x):=\sup\{h(x)\ |\ h\in\cA_{[c,f|_S,M]}\},
$$
as well as its lower envelope, the function $\alpha_{[c,f|_S,M]}:X\to\RX$ defined by
$$
\alpha_{[c,f|_S,M]}(x):=\inf\{h(x)\ |\ h\in\cA_{[c,f|_S,M]}\}.
$$
In fact, if $h:X\to\RX$ is any function such that
\begin{equation}\label{problem}
G(M)\subset G(\partial_c h)\ \ \ \ and \ \ \ h|_S=f|_S,
\end{equation}
 then $\alpha_{[c,f|_S,M]}\leq h$ and $h^c\in\cA_{[c,f^c|_{M(S)},M^{-1}]}$. If $h$ is $c$-convex, then
\begin{equation}\label{conjuA}
h\in\cA_{[c,f|_S,M]}\ \ \Leftrightarrow\ \ \ h^c\in\cA_{[c,f^c|_{M(S)},M^{-1}]}.
\end{equation}
Furthermore,
\begin{equation}\label{conjugamma}
\alpha_{[c,f|_S,M]}^c=\gamma_{[c,f^c|_{M(S)},M^{-1}]}\ \ \ \ and\ \ \ \
\gamma_{[c,f|_S,M]}^c=\alpha_{[c,f^c|_{M(S)},M^{-1}]}.
\end{equation}
\\
In the case where  $S=\mathrm{dom}(M)$, we have
\begin{align}
&\gamma_{[c,f|_{\dom (M)},M]}=(f+\iota_{\dom (M)})^{cc}\label{gammafulldom}\\
and\ \ \ \ \ &\nonumber\\
&\alpha_{[c,f|_{\mathrm{dom}(M)},M]}(x)=\ (f^c+\iota_{\mathrm{Im}(M)})^c(x)=\ \sup_{(s,t)\in G(M)}\ [f(s)+c(x,t)-c(s,t)],\ \ \ \ x\in X.\label{conjufulldom}
\end{align}
In this case, if $h:X\to\RX$ is $c$-convex, then
\begin{equation}\label{fulldomcrit}
h\in\cA_{[c,f|_{\dom (M)},M]}\ \ \ \ \Leftrightarrow\ \ \ \ \ \alpha_{[c,f|_{\dom (M)},M]}\leq h\leq \gamma_{[c,f|_{\dom (M)},M]}.
\end{equation}
\end{theorem}

\begin{proof}
In order to obtain the nonemptines of $\cA_{[c,f|_S,M]}$, we employ Corollary \ref{duality}. Since $f$ is a proper $c$-antiderivative of $M$, the function $f^{cc}$ is a $c$-antiderivative of $M$. Since $S$ is a set of points where $f$ is $c$-subdifferentiable, $f^{cc}$ agrees with $f$ at any point in $S$. The $c$-convexity of $f^{cc}$ now implies that $f^{cc}\in\cA_{[c,f|_S,M]}$. Since $\gamma_{[c,f|_S,M]}|_S=f|_S$, it is the proper upper envelope of $\cA_{[c,f|_S,M]}$. Applying the upper envelope properties \ref{upenvconv} and \ref{upenvantider}, we conclude that the upper envelope $\gamma_{[c,f|_S,M]}$ is a $c$-convex $c$-antiderivative of $M$ and therefore belongs to $\cA_{[c,f|_S,M]}$. Let $h:X\to\RX$ satisfy \eqref{problem}. Then, again, according to Corollary \ref{duality}, $h^c$ is a $c$-convex $c$-antiderivative of $M^{-1}$. Given $t\in M(S)$, let $s\in S$ be such that $t\in M(s)$. Then
\begin{equation}\nonumber
h^c(t)=c(s,t)-h(s)=c(s,t)-f(s)=f^c (t).
\end{equation}
We conclude that $h^c|_{M(S)}=f^c|_{M(S)}$ and consequently that $h^c\in\cA_{[c,f^c|_{M(S)},M^{-1}]}$. Conversely, if $h^c\in\cA_{[c,f^c|_{M(S)},M^{-1}]}$ then $h^{cc}\in \cA_{[c,f^{cc}|_{M^{-1}(M(S))},(M^{-1})^{-1}]}$. If, in addition, $h$ is $c$-convex, then since $(M^{-1})^{-1}=M$, $f^{cc}|_S=f|_S$, $S\subset M^{-1}(M(S))$ and since $h^{cc}=h$, we conclude that $h\in\cA_{[c,f|_S,M]}$. Furthermore, we see that if $h:X\to\RX$ is any function which satisfies  \eqref{problem}, then $h^c\leq\gamma_{[c,f^c|_{M(S)},M^{-1}]}$. Consequently, the function $\gamma^c_{[c,f^c|_{M(S)},M^{-1}]}\in\cA_{[c,f^{cc}|_{M^{-1}(M(S))},(M^{-1})^{-1}]}\subset\cA_{[c,f|_S,M]}$ satisfies $\gamma^c_{[c,f^c|_{M(S)},M^{-1}]}\leq h^{cc}\leq h$. Since $h$ was arbitrary, we conclude that $\gamma^c_{[c,f^c|_{M(S)},M^{-1}]}$ is , in fact, the lower envelope of $\cA_{[c,f|_S,M]}$, that is, we recognize it as $\alpha_{[c,f|_S,M]}$. As a consequence of this discussion and of \eqref{conjuA}, we conclude, in particular, that for any function $h:X\to\RX$,
$$
h^{cc}\in\cA_{[c,f|_S,M]}\ \ \Leftrightarrow\ \ \ h^c\in\cA_{[c,f^c|_{M(S)},M^{-1}]}.
$$
Note that an arbitrary function in $\cA_{[c,f^c|_{M(S)},M^{-1}]}$ may be written as $h^c$ for some function $h\in\cA_{[c,f|_S,M]}$.  Taking now $h\in \cA_{[c,f|_S,M]}$ and $c$-transforming the inequality
$$\alpha_{[c,f|_S,M]}\leq h\leq \gamma_{[c,f|_S,M]}$$
we get
$$\gamma_{[c,f|_S,M]}^c\leq h^c\leq\alpha_{[c,f|_S,M]}^c.$$
Since $h^c$ was an arbitrary function in $\cA_{[c,f^c|_{M(S)},M^{-1}]}$, we conclude that $\gamma_{[c,f|_S,M]}^c$ and $\alpha_{[c,f|_S,M]}^c$ are the lower envelope and upper envelops of $\cA_{[c,f^c|_{M(S)},M^{-1}]}$, respectively. This completes the proof of \eqref{conjugamma} and the proof of the general case of Theorem \ref{main}.

We now consider the case where $S=\dom (T)$. Since $f$ is a $c$-antiderivative of $M$, we may apply Proposition \ref{subinclusion} and conclude that the function $f+\iota_{\dom (M)}$ is the greatest $c$-antiderivative $M$ that agrees with $f$ on $\dom (M)$. It follows that its $c$-convexification $(f+\iota_{\dom (M)})^{cc}$ is the greatest $c$-convex $c$-antiderivative of $M$ that agrees with $f$ on $\dom (M)$, that is , \eqref{gammafulldom} holds. Indeed, if $h$ is a $c$-convex $c$-antiderivative of $M$ that agrees with $f$ on $\dom (M)$, then $h\leq f+\iota_{\dom (M)}$, and consequently, $h=h^{cc}\leq (f+\iota_{\dom (M)})^{cc}$. Since $h$ was an arbitrary element of $\cA_{[c,f|_{\dom (M)},M]}$, we conclude that $(f+\iota_{\dom (M)})^{cc}$ is the greatest function in $\cA_{[c,f|_{\dom (M)},M]}$, which completes the proof of \eqref{gammafulldom}. In order to prove \eqref{conjufulldom} we proceed in a similar way. Since $f^c$ is a $c$-antiderivative of $M^{-1}$, applying Proposition \ref{subinclusion}, we conclude that the function $f^c+\iota_{\im(M)}$ is the greatest antiderivative of $M^{-1}$ that agrees with $f^c$ on $\im(M)$. Hence $(f^c+\iota_{\im(M)})^c\in\cA_{[c,f|_{\dom (M)},M]}$. For a function $h:X\to\RX$ we now have
$$
h\in\cA_{[c,f|_{\dom (M)},M]}\ \ \Rightarrow\ \ \ h^c\in \cA_{[c,f^c|_{\im(M)},M^{-1}]}\ \ \Rightarrow\ \ h^c\leq f^c+\iota_{\im(M)}\ \ \Rightarrow\ \ (f^c+\iota_{\im(M)})^c\leq h^{cc}=h.
$$
Since $h$ was arbitrary, we conclude that $(f^c+\iota_{\im(M)})^c$ is the lower envelope of $\cA_{[c,f|_{\dom (M)},M]}$. This verifies the left equality of  \eqref{conjufulldom}. The right equality of \eqref{conjufulldom} is evident from the following computation:
$$
(f^c+\iota_{\mathrm{Im}(M)})^c(x)=\ \sup_{t\in\mathrm{Im}(M)}\ \ c(x,t)-f^c(t)=\ \ \sup_{(s,t)\in M}\ [c(x,t)+f(s)-c(s,t)].
$$
Finally, in order to prove \eqref{fulldomcrit}, we first note that the implication $\Rightarrow$ is trivial. Conversely, if $\alpha_{[c,f|_{\dom (M)},M]}\leq h\leq \gamma_{[c,f|_{\dom (M)},M]}$, then $h|_{\dom (M)}=f|_{\dom (M)}$. Applying Proposition \ref{subinclusion}, we conclude that $h$ is a $c$-antiderivative of $M$ and consequently that $h\in\cA_{[c,f|_{\dom (M)},M]}$.
\end{proof}

We see that the consequences of Proposition \ref{ccc} play a crucial role in the proof. In particular, the nonemptiness of $\cA_{[c,f|_s,M]}$ and the existence of a minimal member follow from these duality results. A different approach, which does not employ these duality properties of the $c$-transform and which gives rise, constructively, to a minimal member, is presented in the next section. In the particular case where $S=\dom (M)$, such an argument is already at hand.

\begin{remark}
In the case where $S=\dom (T)$, we present a straightforward proof of the nonemptiness of $\cA_{[c,f|_{\dom (M)},M]}$ and of formula (\ref{conjufulldom}) that does not use consequences of the $c$-convexification property.
Define $\alpha:X\to\RX$ by
$$
\alpha(x)= \ \sup_{(s,t)\in M}\ [f(s)+c(x,t)-c(s,t)].
$$
For $x\in\dom (M)$ we see, by choosing $s=x$, that  $\alpha(x)\geq f(x)$. Since $G(M)\subset G(\partial_c f)$, for every $x\in X$ we have
$$
\alpha(x)=\sup_{(s,t)\in M}\ [f(s)+c(x,t)-c(s,t)]\leq f(x).
$$\\
It follows that $\alpha|_{\mathrm{dom}(M)}=f|_{\mathrm{dom}(M)}$. Consequently, for every $(s,t)\in M$ and $x\in X$, we see that
\begin{equation}
\alpha(s)+c(x,t)-c(s,t)\
=\ f(s)+c(x,t)-c(s,t)\ \leq\ \alpha(x),\nonumber
\end{equation}\\
which implies that $G(M)\subset G(\partial_c\alpha)$. Since $\alpha$ is the proper upper envelope of $c$-convex functions and since it satisfies  \eqref{problem}, we conclude that $\alpha\in\cA_{[c,f|_{\dom (M)},M]}$. If $h$ is any function which satisfies \eqref{problem}, then for every $x\in X$,
\begin{equation}
\alpha(x)=\ \sup_{(s,t)\in M}\ [h(s)+c(x,t)-c(s,t)] \leq\ h(x),\nonumber
\end{equation}
which verifies that $\alpha=\alpha_{[c,f|_{\mathrm{dom}(M)},M]}$.
\end{remark}

We remark in passing that \eqref{conjuA} is an extension of (16) from \cite{br}, which dealt with the classical case. Clearly, the converse implication there only holds when $h$ is lower semicontinuous and convex.

Before ending this section we wish to make a remark regarding the closedness of $\cA_{[c,f|_S,M]}$ under nontrivial convex combinations. It is clear that if $g|_S=h|_S=f|_S$, then for any $0<\lambda<1$ we have $(\lambda g+(1-\lambda)h)|_S=f|_S$. However, in general, it may happen that $\cA_{[c,f|_S,M]}$ is not closed under nontrivial convex combinations:

\begin{example}
In the setting of Example \ref{exampleA}, consider the mapping $M$, the set $S=\{x\in\RR\ |\ x\geq 0\}$ and the function $f:X\to\RX$ defined by  $f(x)=x$. Then
$$
\alpha_{[c,f|_S,M]}(x)=x,\ \ \ \gamma_{[c,f|_S,M]}(x)=|x|\ \ \ \mathrm{and}\ \ \ \cA_{[c,f|_S,M]}=\big\{x,\ |x-p|+p\ \ \big|\ p\leq 0\big\}.
$$
Observe that in this example any $c$-convex function $h$ such that $h|_S=f|_S$ is automatically a $c$-antiderivative of $M$, that is, $\cA_{[c,f|_S,M]}=\Gamma_c(X)\cap\{h:X\to\RX\ \big|\ h|_S=f|_S\}$.\\
We have already seen that any nontrivial convex combination of $x$ and $|x|$ is not $c$-convex. Hence, $\cA_{[c,f|_S,M]}$ is not closed under nontrivial convex combinations in this case.
\end{example}

Finally, recalling Proposition \ref{convcomb}, we note the following Corollary.

\begin{corollary}\label{convA}
Let $g,h\in \cA_{[c,f|_S,M]}$ and suppose that $0<\lambda<1$. Then $\lambda g+(1-\lambda)h\in \cA_{[c,f|_S,M]}$ as soon as $\lambda g+(1-\lambda)h$ is $c$-convex.
\end{corollary}

\section{Cyclic Monotonicity and Minimal Antiderivatives}

We begin this section by recalling the following notions.

\begin{definition}[c-cyclic monotonicity and c-monotonicity]\label{cycmondef}
A mapping $M:X\rightrightarrows Y$ is said to be cyclically monotone of order $n$ with respect to $c$, $n$-$c$-monotone for short, when given any set of $n$ ordered pairs $\{(x_i,y_i)\}_{i=1}^n\subset G(M)$, if we set $x_{n+1}=x_1$, then
\begin{equation}\label{cycmonres}
0\leq\sum_{i=1}^n [c(x_i,y_i)-c(x_{i+1},y_i)].
\end{equation}
In this case we say that $G(M)$ is an $n$-$c$-monotone set. A mapping $M$ is said to be cyclically monotone with respect to $c$, $c$-cyclically monotone for short, if it is $n$-$c$-monotone for all $n\in\mathbb{N}$. A $2$-$c$-monotone mapping is simply called a $c$-monotone mapping. Explicitly, the mapping $M$ is $c$-monotone if for all $(x_1,y_1),(x_2,y_2)\in G(M)$, we have
\begin{equation}\label{mondef}
0\leq c(x_1,y_1)-c(x_1,y_2)-c(x_2,y_1)+c(x_2,y_2).
\end{equation}
The mapping $M$ is said to be maximal $n$-$c$-cyclically monotone if $G(M)$ has no proper $n$-$c$-cyclically monotone extension in $X\times Y$.
\end{definition}

We proceed by recalling Rockafellar's characterization of cyclically monotone
mappings. This characterization also holds in the setting of $c$-monotonicity and so does its
proof.

\begin{definition}[Rockafellar's antiderivative]
With the function $c$, the mapping $M:X\rightrightarrows Y$ and the point $s\in\dom (M)$, we associate Rockafellar's function $R_{[c,M,s]}:X\to\RX$, defined by
\begin{equation}
R_{[c,M,s]}(x):=
\sup_{\begin{array}{c}
                       n\in\mathbb{N},\\
                       x_1=s,\ x_{n+1}=x,\\
                       \{(x_i,y_i)\}_{i=1}^n\subset G(M)
                     \end{array} }
\ \ \sum_{i=1}^n [c(x_{i+1},y_i)-c(x_i,y_i)].
\end{equation}
\end{definition}

\begin{theorem}
A proper mapping $M:X\rightrightarrows Y$ is $c$-cyclically monotone if and only if it
has a proper $c$-antiderivative. In this case, in particular, for any $s\in\dom (M)$,
Rockafellar's function $R_{[c,M,s]}$ is a proper $c$-convex $c$-antiderivative
of $M$ which satisfies $R_{[c,M,s]}(s)=0$. In fact, $R_{[c,M,s]}$ is proper if and only if $M$ is proper and $c$-cyclically monotone.
\end{theorem}

\begin{proof}
Suppose first that $f:X\to\RX$ is a $c$-antiderivative of $M$, that is, $G(M)\subset G(\partial_c f)$. Let $\{(x_i,y_i)\}_{i=1}^n\subset G(M)$ and let $x_{n+1}=x_1$. For every $1\leq i\leq n$, we have
$$
c(x_{i+1},y_i)-c(x_i,y_i)\leq f(x_{i+1})-f(x_i).
$$
Summing up these inequalities over $i$, we see that $M$ is $c$-cyclically
monotone. Conversely, suppose that $M$ is proper and $c$-cyclically monotone and let $s\in\dom (M)$. For every $\{(x_i,y_i)\}_{i=1}^n\subset G(M)$ such that $x_1=s$, let $x_{n+1}=x_1=s$. Since $\sum_{i=1}^n [c(x_{i+1},y_i)-c(x_i,y_i)]\leq 0$, from the definition of Rockafellar's function we get $R_{[c,M,s]}(s)\leq 0$. Letting $n=1$ we obtain $R_{[c,M,s]}(s)=0$. Since $R_{[c,M,s]}$ is now seen to be proper, it is $c$-convex as a proper upper envelope of $c$-convex functions. Let $(x,y)\in G(M)$ and $x'\in X$. For every $\{(x_i,y_i)\}_{i=1}^n\subset G(M)$ such that $x_1=s$, from the definition of $R_{[c,M,s]}$ we have
$$
\sum_{i=1}^{n-1} [c(x_{i+1},y_i)-c(x_i,y_i)]+c(x,y_n)-c( x_n,y_n)+c(x',y)-c( x,y)\leq R_{[c,M,s]}(x'),
$$
which implies that
$$
R_{[c,M,s]}(x)+c(x',y)-c( x,y)\leq R_{[c,M,s]}(x').
$$
Since $x'$ was arbitrary, we get $(x,y)\in G(\partial_c R_{[c,M,s]})$ and consequently, that $R_{[c,M,s]}$ is a $c$-antiderivative of $M$. Finally, if $M$ is not
$c$-cyclically monotone, then there is a set $\{(x_i,y_i)\}_{i=1}^k\subset G(M)$ such that if we let $x_{k+1}=x_1$, then $q:=\sum_{i=1}^k [c(x_{i+1},y_i)-c(x_i,y_i)] > 0$. For each $m\in\mathbb{N}$ and $1\leq j\leq mk+1$, we define $(u_j,v_j)=(x_i,y_i)$ when $j\mod k=i$. Let $t\in M(s)$. For every $x\in X$, we have
\begin{align}
R_{[c,M,s]}(x)\geq &\ c(u_1,t)-c(s,t)+\sum_{j=1}^{mk}[c(u_{i+1},v_i)-c(u_i,v_i)]+c(x,v_{mk+1})-c(u_{mk+1},v_{mk+1})\nonumber\\
=&\ c(x_1,t)-c(s,t)+mq+c(x,y_1)-c(x_1,y_1)\underset{m\rightarrow\infty}{\longrightarrow} \infty,\nonumber
\end{align}
which means that $R_{[c,M,s]}$ is not proper in this case.
\end{proof}

We are now ready to employ Rockafellar's antiderivative in order to reestablish the nonemptiness of $\cA_{[c,f|_S,M]}$ by explicitly constructing the function $\alpha_{[c,f|_S,M]}$. Using duality relations from the previous section, this approach enables us to get explicit formulae for $\gamma_{[c,f|_S,M]}$ as well. This will be carried out, in particular, in the context of an application in the next section.

\begin{theorem}\label{mainformula}
Suppose that $f:X\to\RX$ is a $c$-antiderivative of the mapping $M:X\rightrightarrows Y$ and suppose that $\emptyset\neq S\subset\mathrm{dom} (M)$. Then the minimal $c$-antiderivative of $M$ that equals $f$ at the points of $S$, the function $\alpha_{[c,f|_S,M]}\in\cA_{[c,f|_S,M]}$, is given by
\begin{equation}\label{alpha}
\alpha_{[c,f|_S,M]}(x)=\ \sup_{s\in S}\ [f(s)+R_{[c,M,s]}(x)]\ \ \ \ \ \ \ \forall x\in X.
\end{equation}
\end{theorem}

\begin{proof}
We define $\alpha:X\to\RX$ by
$$
\alpha(x)=\ \sup_{s\in S}\ [f(s)+R_{[M,s]}(x)],\ \ \ \ x\in X.
$$
First we consider the case where $S=\{s\}$ is a singleton. In order to prove this case we may assume that $f(s)=0$. In this case we have $\alpha=R_{[c,M,s]}$. Consequently, $\alpha$ is a proper $c$-convex $c$-antiderivative of $M$ and $\alpha|_S=f|_S=f(s)=0$, that is, $\alpha\in\cA_{[c,f|_S,M]}$. In order to see that $\alpha$ is the lower envelope of $\cA_{[c,f|_S,M]}$, we prove the following minimality property of $R_{[c,M,s]}$: it is the minimal $c$-antiderivative of $M$ that vanishes at $s$. That is, for $h:X\to\RX$, we claim that
\begin{equation}\label{rocminprop}
 G(M)\subset G(\partial_c h)\ \ \mathrm{and}\ \ \ h(s)=0\ \ \ \Rightarrow\ \ \ \ R_{[c,M,s]}\leq h.
\end{equation}\\
Indeed, let $x\in X$. For $n\in \mathbb{N},\ x_1=s$ and $\{(x_1,y_1),\dots,(x_n,y_n)\}\subset G(M)\subset G(\partial_c h)$, we have
\begin{align}\label{minineq}
     &\sum_{i=1}^{n-1}[c(x_{i+1},y_i)-c(x_i,y_i)]+\ c(x,y_n)-c(x_n,y_n) \\
\leq &\sum_{i=1}^{n-1}[h(x_{i+1})-h(x_i)]+h(x)-h(x_n)\nonumber\\
   = &\ h(x)-h(x_1)=h(x)-h(s)=h(x), \nonumber
\end{align}
which implies that $R_{[c,M,s]}(x)\leq h(x)$, as claimed.

In the general case, for every $s\in S$, we define $\beta_s:X\to\RX$ by
$$
\beta_s(x)=f(s)+R_{[c,M,s]}(x).
$$
Let $h:X\to\RX$ be a function such that $G(M)\subset G(\partial_c h)$ and $h|_S=f|_S$. Then for every $s\in S$, we have
$$
\begin{array}{cc}
  (a)\ \ \beta_s\ \mathrm{is}\ c-\mathrm{convex}; & (b)\ \ G(M)\subset G(\partial_c \beta_s);\\
  \\
  (c)\ \ \beta_s\leq h;\ \ \ \ \ \ \ \ \ \ \ \ & \ \ \ \ (d)\ \ \beta_s(s)=h(s)=f(s).
\end{array}
$$
Parts $(a)$ and $(b)$ are clear because the addition of a constant does not affect the $c$-subdifferential and $c$-convexity. Since $R_{[c,M,s]}(s)=0$, $(d)$ is also clear. Part $(c)$ is a straightforward application of our particular case above, that is, $\beta_s(\cdot)=f(s)+R_{[c,M,s]}(\cdot)$ is the minimal $c$-antiderivative of $M$ that equals $f$ at $s$. Finally, from $(a)$ and $(b)$ it follows that $\{\beta_s\}$ is a family of $c$-convex $c$-antiderivatives of $M$. Part $(c)$ implies that the upper envelope $\alpha=\sup_s \beta_s$ is proper. Applying Propositions \ref{upenvconv} and \ref{upenvantider}, we conclude that $\alpha$ is a  $c$-convex $c$-antiderivative of $M$. From $(c)$ we had $\alpha\leq h$, now $(d)$ implies that $\alpha|_S=f|_S$. We conclude that $\alpha\in\cA_{[c,f|_S,M]}$ and satisfies $\alpha\leq h$. Since $h$ was an arbitrary $c$-antiderivative of $M$ such that $h|_S=f|_S$, we conclude that $\alpha=\alpha_{[c,f|_S,M]}$, as asserted.

\end{proof}

\begin{remark}
If $S=\dom (M)$, then formula (\ref{alpha}) reduces to formula (\ref{conjufulldom}) as follows: let $x\in X$. For $n\in \mathbb{N},\ x_1=s$ and $\{(x_1,y_1),\dots,(x_n,y_n)\}\subset G(M)\subset G(\partial_c f)$, we have seen in \eqref{minineq} that
$$
\sum_{i=1}^{n-1}[c(x_{i+1},y_i)-c(x_i,y_i)]+\ c(x,y_n)-c(x_n,y_n)\leq f(x_n)-f(s)+c(x,y_n)-c(x_n,y_n).
$$
Therefore
$$
\sup_{(s,t)\in M}\ [f(s)+c(x,t)-c(s,t)]\leq \sup_{s\in S}\ [f(s)+R_{[c,M,s]}(x)]\leq\sup_{(x_n,y_n)\in M}\ [f(x_n)+c(x,y_n)-c(x_n,y_n)],
$$
that is, we have equality in both inequalities.
\end{remark}

\section{Optimal Antiderivatives as Constrained Optimal Lipschitz Extensions }

In our discussion of Lipschitz functions we will assume that the functions are $1$-Lipschitz. Since $Kd^\alpha$, where $0<\alpha\leq 1$ and $K>0$, is also a metric whenever $d$ is, our results and formulae are also easily extensible to $\alpha$-H\"{o}lder continuous functions with constant $K$ by replacing $d$ with $Kd^\alpha$. First we discuss the Lipschitz extension problem without introducing abstract convexity into the discussion. We begin with the following envelope properties of sets of Lipschitz functions.

\begin{proposition}\label{lipenv}
Let $(X,d)$ be a metric space. Let $M:X\rightrightarrows X$ and suppose that $\{h_s\}$ is a family of functions such that $h_s:X\to\RR$ is $1$-Lipschitz and 
\begin{equation}\label{cond1}
h_s(y)-h_s(x)=d(x,y)\ \ \ \ \ \ \ \ \ \ \ \ \forall(x,y)\in G(M)
\end{equation}
for every $s$. Then $\sup_s h_s$ and $\inf_s h_s$ are $1$-Lipschitz functions which also satisfy (\ref{cond1}), when proper. In particular, if $\{h_s\}$ is any family of $1$-Lipschitz functions, then $\sup_s h_s$ and $\inf_s h_s$ are $1$-Lipschitz functions, when proper.
\end{proposition}
We omit the direct proof. We will offer a proof in the setting of $-d$-convexity shortly. We note in passing  that the particular case will follow from the general case when we let $M=I$, where $I:X\to X$ is the identity on $X$. Next, we recall the following well-known extension theorem of McShane \cite{mcs} and Whitney \cite{whi}.

\begin{theorem}
Let $S$ be a nonempy subset of a metric space $(X,d)$ and let $f:S\to\mathbb{R}$ be $1$-Lipschitz. Then $f$ extends to a $1$-Lipschitz function which is defined on all of $X$. In particular, the functions
$$
\alpha(x)=\sup_{s\in S}\ [f(s)-d(x,s)]\ \ \ \ \ \ \ \ \ \mathrm{and}\ \ \ \ \ \ \ \ \ \gamma(x)=\inf_{s\in S}\ [f(s)+d(x,s)]
$$
are $1$-Lipschitz extensions of $f$. If $h:X\to\mathbb{R}$ is $1$-Lipschitz and $h|_S=f$, then $\alpha\leq h\leq\gamma$.
\end{theorem}

\begin{proof}
If $h$ is a $1$-Lipschitz extension of $f$, then for $x\in X$ and every $s\in S$, we have
$$
f(s)-d(x,s)=h(s)-d(x,s)\ \leq\ h(x)\ \leq\ h(s)+d(x,s)=f(s)+d(x,s).
$$
It follows that $\alpha\leq h\leq\gamma$. Since $f$ is $1$-Lipschitz on $S$, we have $\alpha|_S=\gamma|_S=f|_S$. In particular, $\alpha$ and $\gamma$ are proper. Now Proposition \ref{lipenv} guarantees that $\alpha$ and $\gamma$ are $1$-Lipschitz.
\end{proof}

These facts suffice to establish the existence of optimal extensions for the following problem.

\begin{theorem}\label{lipext1}
Let $(X,d)$ be a metric space. Let $f:A\subset X\to\RR$ be a $1$-Lipschitz function and let $M:A\rightrightarrows A$ be a mapping such that
\begin{equation}\label{cond2}
f(y)-f(x)=d(x,y)\ \ \ \ \ \ \mathrm{for\ all}\ \ (x,y)\in G(M).
\end{equation}
Given $\emptyset\neq S\subset\dom (M)$, $f|_S$ extends to a $1$-Lipschitz function which is defined on all of $X$ and which satisfies (\ref{cond2}). In particular, there exist minimal and maximal $1$-Lipschitz extensions of $f|_S$ which are defined on $X$ and satisfy (\ref{cond2}).
\end{theorem}

\begin{proof}
Since $f$ extends to a $1$-Lipschitz function which is defined on $X$, this extension clearly extends $f|_S$ and satisfies (\ref{cond2}). Thus, the set of $1$-Lipschitz extensions of $f|_S$ which satisfy (\ref{cond2}) is not empty and has proper lower and upper envelopes. From Proposition \ref{lipenv} it follows that these envelopes belong to the set.
\end{proof}

Now we introduce abstract convexity into our discussion in order to recover all the above facts from a different perspective. This will allow us to construct the optimal extensions explicitly. We do not assume the McShane and Whitney extension theorem; it follows as a particular case. When $(X,d)$ is a metric space, $Y=X$ and we take the function $c:X\times Y\to\RR$ to be $c=-d$, the following relations between $c$-convex functions and $1$-Lipschitz functions emerge.

\begin{proposition}\label{lipconv}
Let $(X,d)$ be a metric space. For any proper function $f:X\to\RX$, the following assertions are equivalent:\\
\\
1. $f$ is $1$-Lipschitz;\\
2. $f^{-d}=-f$;\\
3. $f$ is $-d$-convex;\\
4. $f$ is a $-d$-antiderivative of the identity $I:X\to X$. That is, $G(I)\subset G(\partial_{-d} f)$.\\
\\
In this case, it follows that $f:X\to\RR$. The graph of the $-d$-subdifferential of $f$ is the set of ordered pairs $(x,y)\in X\times X$ such that $f$ preserves the distance $d(x,y)$ in the sense that $f(y)-f(x)=d(x,y)$.
\end{proposition}

\begin{proof}
If $f$ is $1$-Lipschitz, then for every $y\in X$ we have $f^{-d}(y)=\sup_{x\in X}\ -d(x,y)-f(x)\leq -f(y)$. Since equality holds for $x=y$, we have verified that $1\Rightarrow 2$. If $f^{-d}=-f$, then for every $x\in X$,
$$
f(x)\geq (f^{-d})^{-d}(x)=\sup_{y\in X}[-d(x,y)-f^{-d}(y)]=\sup_{y\in X}[-d(x,y)+f(y)]\geq f(x),
$$
which implies that $f=(f^{-d})^{-d}$ and consequently that $f$ is $-d$-convex. This verifies $2\Rightarrow 3$. For every $y\in X$, the function $-d(\cdot,y)$ is a $-d$-antiderivative of the identity $I$. Indeed, let $x\in X$. For any $x'\in X$ we have
$$
-d(\cdot,y)(x)-d(x',x)\leq -d(\cdot,y)(x')=-d(\cdot,y)(x')-d(x,x),
$$
which implies that $x\in\partial_{-d}(-d(\cdot,y))(x)$. Since this holds for every $x\in X$, we obtain $G(I)\subset G(\partial_{-d}(-d(\cdot,y)))$. It follows that $-d(\cdot,y)+const$ is also a $-d$-antiderivative of the identity $I$ for any constant. Now suppose that $f$ is $-d$-convex, say $f=g^{-d}=\sup_{y\in X}\ -d(\cdot,y)+g(y)$. Then $f$ is the proper upper envelope of a family of $-d$-antiderivatives of $I$. It now follows from Proposition \ref{upenvantider} that $f$ is also a $-d$-antiderivative of $I$, which concludes the proof of $3\Rightarrow 4$. If $f$ is a $-d$-antiderivative of $I$, then for every $x,x'\in X$, we have $x\in\partial_{-d} f(x)$ and consequently we obtain
$$
f(x)-f(x')\leq d(x,x')-d(x,x)=d(x,x'),
$$
that is, $4\Rightarrow 1$.
\end{proof}

\begin{remark}
Since we have not used the fact that $d(x,y)=0$ only if $x=y$, Proposition \ref{lipconv} also holds for a pseudometric $d$. The fact that the identity mapping is $-d$-cyclically monotone is an immediate consequence of the properties of the metric $d$.
\end{remark}

It follows that the identity mapping is the most trivial $-d$-subdifferential of $-d$-convex functions. It is another example of a situation which is impossible in classical convex analysis: the $c$-subdifferential of a $c$-convex function need not be maximal $c$-cyclically monotone or maximal $c$-monotone. Furthermore, it does not determine its $c$-antiderivative up to an additive constant. In the classical setting, where proper, convex and lower semicontinuous functions are precisely the $c$-convex functions, when we take $c$ to be the pairing between a Banach space and its dual, these are well known facts due to Rockafellar \cite{roc,roc2}. That is, the subdifferential of a proper, convex and lower semicontinuous function determines its antiderivative up to an additive constant. It is maximal monotone and maximal cyclically monotone. This is illustrated by the following example which extends Example 2.3 from \cite{rol}.

\begin{example}
The identity mapping $I$ on a metric space is $-d$-cyclically monotone. $G(I)$ is contained in the graph of the $-d$-subdifferential of any $-d$-convex function. As a consequence, any $-d$-cyclically monotone set $\ G(M)\subset X\times X$ extends to the $-d$-cyclically monotone set $G(M)\cup G(I)$. In particular, $G(I)$ is extensible by one point to a $-d$-cyclically monotone set, using any point in $X\times X$. Therefore, if $X$ is not a singleton, $I$ is not maximal $-d$-cyclically monotone and it is not maximal $-d$-monotone. However, it is the (whole) $-d$-subdifferential of any function $f:X\to\RR$ such that $f(y)-f(x)<d(x,y)$ for all $x\neq y$. In particular, it is the $-d$-subdifferential of any $K$-Lipschitz function, where $0\leq K<1$. Therefore, $I$ does not determine its $-d$-convex $-d$-antiderivative up to an additive constant.
\end{example}

\begin{proof}
The fact that $I$ is $-d$-cyclically monotone and the fact that $G(I)$ is a subset of $G(\partial_{-d}f)$ for any $-d$-covex function $f$ are clear from Proposition \ref{lipconv}. If $G(M)$ is a $-d$-cyclically monotone set, then $M$ admits a $-d$-convex $-d$-antiderivative $f$. It follows that $G(M)\cup G(I)\ \subset\ G(\partial_{-d}f)$, which is cyclically monotone. Since a singleton in $X\times X$ is always $-d$-cyclically monotone, $G(I)$ extends to a $-d$-cyclically monotone set by unifying it with any singleton. If $f:X\to\RR$ is such that $f(y)-f(x)<d(x,y)$ for all $x\neq y$, then $f$ is $-d$-convex. In this case,
$$
G(\partial_{-d}f)=\{(x,y)\in X\times X\ |f(y)-f(x)=d(x,y)\}=G(I).
$$
\end{proof}

Having established Proposition \ref{lipconv}, we now show that the envelope properties of Lipschitz functions in Proposition \ref{lipenv} are a direct consequence of this proposition, when we use $c$-convexity properties of envelopes:

\begin{proof}(of Proposition \ref{lipenv})
Let $G(M)\subset X\times X$ be a set and suppose that $\{h_s\}$ is a set of $1$-Lipschitz functions such that for every $s$,
$$
h_s(y)-h_s(x)=d(x,y)\ \ \ \ \ \ \ \ \ \ \ \ \forall(x,y)\in G(M).
$$
Then according to Proposition \ref{lipconv}, the function $h_s$ is a $-d$-convex $-d$-antiderivative of $M$ for every $s$. From Propositions \ref{upenvconv} and \ref{upenvantider} it follows that when $\gamma:=\sup_s h_s$ is proper, it is also a $-d$-convex $-d$-antiderivative of $M$, that is, $\gamma$ is $1$-Lipschitz and satisfies $\gamma(y)-\gamma(x)=d(x,y)$ for every $(x,y)\in G(M)$. In order to see the lower envelope properties we apply the same argument with the upper envelope of the $1$-Lipschitz functions $-h_s$.
\end{proof}

Next, we turn to the main result of this section. We reformulate the hypotheses of Theorem \ref{lipext1}. Both in Theorem \ref{lipext1} and in Theorem \ref{Lipext}, the hypotheses provide us with a $-d$ antiderivative $f$ of the mapping $M:X\rightrightarrows X$ and a nonempty subset $S$ of $\dom (M)$. Therefore the existence of the optimal extensions of $f|_S$ which are $-d$-convex $-d$-antiderivatives of $M$ and their formulae hold in both cases. Unlike in Theorem \ref{lipext1}, in Theorem \ref{Lipext} the reformulation does not require $f$ to be $1$-Lipschitz outside $\dom (M)$. However, it is clear that any $1$-Lipschitz extension of $f|_S$ that satisfies these hypotheses is also uniquely determined on $M(S)$ by the equality
\begin{align}
f(t)-f(s)=d(s,t)\ \ \ \  \mathrm{for\ every}\ s\in S\ \ \mathrm{and}\ \ t\in M(s).\nonumber\\
\nonumber
\end{align}

\begin{theorem}\label{Lipext}
Let $(X,d)$ be a metric space. Let $M:\dom(M)\subset X\rightrightarrows X$ and $f:\dom(M)\to\RR$ satisfy
\begin{equation}\label{dsubM}
f(x)-f(x')\leq d(x',y)-d(x,y)\ \ \ \ for\ all\ \ (x,y)\in G(M)\ \ and\ \ x'\in\dom(M).
\end{equation}
Let $\emptyset\neq S\subset\dom(M)$. Then $f|_S$ extends to a $1$-Lipschitz function $h:X\to\RR$ which satisfies
\begin{equation}\label{dsub}
h(x)-h(x')\leq d(x',y)-d(x,y)\ \ \ \ for\ all\ \ (x,y)\in G(M)\ \ and\ \ x'\in X.
\end{equation}
The set $\cA_{[-d,f|_S,M]}$ of all such extensions $h$ of $f|_S$ is convex. In particular, the function $\alpha_{[-d,f|_S,M]}:X\to\RR$ defined by
\begin{equation}\label{minfSM}
\alpha_{[-d,f|_S,M]}(x)=\sup_{\begin{array}{c}
                       s\in S,\\
                       n\in\mathbb{N},\ x_1=s,\\
                       \{(x_i,y_i)\}_{i=1}^n\subset G(M)\\
                     \end{array} }
f(s)+\sum_{i=1}^{n-1}[ d(x_i,y_i)-d(x_{i+1},y_i)]+d( x_n,y_n)-d(x,y_n)
\end{equation}
is the minimal $1$-Lipschitz function that agrees with $f$ on $S$ and satisfies \eqref{dsub}. The function $\gamma_{[-d,f|_S,M]}:X\to\RR$ defined by
\begin{equation}\label{maxfSM}
\gamma_{[-d,f|_S,M]}(x)=\inf_{\begin{array}{c}
                       s\in S,\\
                       n\in\mathbb{N},\ x_1=s,\\
                       \{(x_i,y_i)\}_{i=1}^n\subset G(M)\\
                     \end{array} }
f(s)+\sum_{i=1}^{n-1} [d(x_i,y_{i+1})-d(x_{i+1},y_{i+1})]+d(x_n,x)
\end{equation}
is the maximal $1$-Lipschitz function that agrees with $f$ on $S$ and satisfies \eqref{dsub}. If $S=\dom(M)$ then
\begin{equation}\label{minfM}
\alpha_{[-d,f|_{\mathrm{dom}(M)},M]}(x)=\sup_{(s,t)\in G(M)}\ [f(s)+d(s,t)-d(x,t)]
\end{equation}
and
\begin{equation}\label{maxfM}
\gamma_{[-d,f|_{\mathrm{dom}(M)},M]}(x)=\inf_{s\in\dom (M)}\ [f(s)+d(x,s)],\ \ \ \ x\in X.
\end{equation}
In particular, suppose that $f:S\subset X\to X$ is $1$-Lipschitz. Then $M=I_S$, where $I_S$ is the identity mapping on $S$,  satisfies \eqref{dsubM}. As a consequence $\mathrm{(McShane,\ Whitney)}$,
\begin{equation}\label{minmaxfSI}
\alpha_{[-d,f|_S,I_S]}(x)=\sup_{s\in S}\ [f(s)-d(x,s)]\ \ \ \ \ \ \ and\ \ \ \ \ \ \ \gamma_{[-d,f|_S,I_S]}(x)=\inf_{s\in S}\ [f(s)+d(x,s)]
\end{equation}
are the minimal and maximal $1$-Lipschitz extensions of $f$, respectively.
\end{theorem}

\begin{proof}
We let $f(x')=\infty$ for $x'\in X\setminus\dom(M)$. Now it follows that for every $(x,y)\in G(M)$ and every $x'\in X$ we have
$$
f(x)-d(y,x')\leq f(x')-d(y,x),
$$
that is, the function $f:X\to\RX$ is a $-d$-antiderivative of $M$. This aligns us within the settings of Theorems \ref{main} and \ref{mainformula}. We conclude that the minimal $-d$-antiderivative of $M$ that agrees with $f$ on $S$ is the function in \eqref{minfSM}. Since $\alpha_{[-d,f|_S,M]}$ is $-d$-convex on $X$, it is $1$-Lipschitz there. Furthermore, $\cA_{[-d,f|_S,M]}$ is precisely the set of $1$-Lipschitz functions $h:X\to\RR$ that satisfy $h|_S=f|_S$ and \eqref{dsub}. Therefore $\alpha_{[-d,f|_S,M]}$ is the minimal function with these properties. Since $\Gamma_{-d}(X)$ is convex, $\cA_{[-d,f|_S,M]}$ is also convex (see Corollary \ref{convA}). In the rest of the proof we verify formulae \eqref{maxfSM}-\eqref{minmaxfSI}. To this end, we may assume that $f$ is $-d$-convex on $X$, that is, $f$ is $1$-Lipschitz on $X$. Indeed, if $h\in\cA_{[-d,f|_S,M]}$, then $\cA_{[-d,f|_S,M]}=\cA_{[-d,h|_S,M]}$ and consequently, $\cA_{[-d,f^{-d}|_{M(S)},M^{-1}]}=\cA_{[-d,h^{-d}|_{M(S)},M^{-1}]}$. Thus we may replace $f$ with $h$, if need be. Now we derive the formula for the function $\alpha_{[-d,f^{-d}|_{M(S)},M^{-1}]}$:
\begin{align}
 &\alpha_{[-d,f^{-d}|_{M(S)},M^{-1}]}(x)=\alpha_{[-d,-f|_{M(S)},M^{-1}]}(x)\nonumber\\
 &=\sup_{\begin{array}{c}
                       t\in M(S),\\
                       n\in\mathbb{N},\ y_1=t,\\
                       \{(x_i,y_i)\}_{i=1}^n\subset G(M)\\
                     \end{array} }
-f(t)+\sum_{i=1}^{n-1} [d(x_i,y_i)-d(x_i,y_{i+1})]+d( x_n,y_n)-d(x,x_n)\label{f(s)ins f(t)}\\
&=\sup_{\begin{array}{c}
                       s\in S,\\
                       n\in\mathbb{N},\ y_1\in M(s),\\
                       \{(x_i,y_i)\}_{i=1}^n\subset G(M)\\
                     \end{array} }
-f(s)+\sum_{i=1}^{n-1} [d(x_{i+1},y_{i+1})-d(x_i,y_{i+1})]-d(x_n,x).\label{f(s)ins f(t)2}
\end{align}
In the first equality we used the fact that $f^{-d}=-f$. In \eqref{f(s)ins f(t)2} we used the fact that for each $t\in M(S)$, $f(t)$ is determined by the equation $f(t)-f(s)=d(s,t)$ for every $s\in M^{-1}(t)$. In particular, for $y_1=t$ and $\{(x_i,y_i)\}_{i=1}^n\subset G(M)$, we let $s=x_1$. In order to obtain \eqref{maxfSM}, we apply \eqref{conjugamma} as follows
\begin{equation}\label{gamma-alpha}
\gamma_{[-d,f|_S,M]}=(\gamma_{[-d,f|_S,M]}^{-d})^{-d}=\alpha_{[-d,f^{-d}|_{M(S)},M^{-1}]}^{-d}=-\alpha_{[-d,f^{-d}|_{M(S)},M^{-1}]}.
\end{equation}
Together with \eqref{f(s)ins f(t)2}, this implies implies \eqref{maxfSM}.

If $S=\dom(M)$, we apply \eqref{conjufulldom} and arrive at \eqref{minfM}. Since we assume that $f$ is a $-d$-convex $-d$-antiderivative of $M$, it follows from \eqref{conjufulldom} that
$$
\alpha_{[-d,f^{-d}|_{\mathrm{Im}(M)},M^{-1}]}(x)=\sup_{(s,t)\in G(M)}\ [-f(t)+d(s,t)-d(s,x)]=\sup_{s\in S}\ [-f(s)-d(s,x)].
$$
Since by \eqref{gamma-alpha} we have $\gamma_{[-d,f|_S,M]}=-\alpha_{[-d,f^{-d}|_{\mathrm{Im}(M)},M^{-1}]}$, \eqref{maxfM} follows. If $f:S\subset X\to X$ is $1$-Lipschitz, then $M=I_S$ satisfies \eqref{dsubM}. The set of $1$-Lipschitz extensions of $f$ is precisely the set of $-d$-antiderivatives of $I_S$ that agree with $f$ on $S$, that is $\cA_{[-d,f|_S,I_S]}$. The formulae for the envelope of $\cA_{[-d,f|_S,I_S]}$ now follow directly from \eqref{minfM} and \eqref{maxfM}.
\end{proof}
\\

\begin{remark}
The formulae for the minimal and maximal extensions do not appear to be symmetric. This is because for both formulae we first used the formula for the minimal extension, and then, to obtain the maximal extension, we used duality. However, in both cases, the initial points $s$ where taken from the given set $S$. This required modifications in the formula for the minimal extension in the dual problem. Keeping in mind that the values $h(t)$ for $t\in M(S)$ are uniquely determined for any extension $h\in\cA_{[-d,f|_S,M]}$ by the equality $h(t)=d(s,t)+h(s)=d(s,t)+f(s)$ for every $s\in M^{-1}(t)$, we can express the maximal extension using the values $f(t)$ of the initial points $t\in M(S)$. Allowing this, we see that the negative of  formula \eqref{f(s)ins f(t)}, which is the formula for $\gamma_{[-d,f|_S,M]}$, and formula \eqref{minfSM} for $\alpha_{[-d,f|_S,M]}$ are symmetric. In the case where $S=\dom (M)$, we have
$$
\alpha_{[-d,f|_{\mathrm{dom}(M)},M]}(x)=\sup_{(s,t)\in G(M)}\ [f(s)+d(s,t)-d(x,t)]=\sup_{t\in\im (M)}[f(t)-d(x,t)],
$$
while
$$
\gamma_{[-d,f|_{\mathrm{dom}(M)},M]}(x)=\inf_{s\in\dom (M)}\ [f(s)+d(x,s)].
$$
We see that in this case, the minimal extension $\alpha_{[-d,f|_{\mathrm{dom}(M)},M]}$ is the minimal McShane and Whitney extension of $f|_{\im(M)}$, while the maximal extension $\gamma_{[-d,f|_{\mathrm{dom}(M)},M]}$ is the maximal McShane and Whitney extension of $f|_{\dom (M)}$. In particular, in any case when $\dom (M)=\im (M)$, it follows that $\alpha_{[-d,f|_{\mathrm{dom}(M)},M]}$ and $\gamma_{[-d,f|_{\mathrm{dom}(M)},M]}$ are precisely the minimal and maximal McShane and Whitney extensions of $f|_{\dom (M)}$, respectively.
\end{remark}

A remark regarding the history of some parts of our discussion in this section is in order. Some of the previously published results regarding the subject matter of this section were stated in settings where $X=\mathbb{R}^n$ and some in the generality of metric spaces. Some have been stated for Lipschitz functions and some for H\"{o}lder continuous functions. As we remarked at the beginning of this section, we do not distinguish between these here. The fact that the $-d$-convex functions are precisely the $1$-Lipschitz functions had already appeared in \cite{evemaa}. It has evolved \cite{leg} into a pointwise variant of the characterization $1\Leftrightarrow 2\Leftrightarrow 3 \Leftrightarrow 4$ of Proposition \ref{lipconv}. This evolvement was recalled later in \cite{sin}. Recently, this has been an example of a coupling function $c$ and the corresponding set of $c$-convex functions which is given by different authors; see, for example, \cite{rol,vil}. Regarding the extension of functions, given a function $f:S\subset X\to\mathbb{R}$, we set $f(x)=\infty$ for $x\in X\setminus S$. In \cite{hir} the McShane and Whitney theorem was reestablished with the aid of the infimal convolution from classical convex analysis in $\mathbb{R}^n$. Given $x\in\mathbb{R}^n$, the value of the maximal extension of $f$ at $x$ is then the infimal convolution of $f$ with the norm. In \cite{leg}, $c$-conjugation theory was applied with the function $c=-d$. Then, the maximal extension of McShane and Whitney was recognized as $(f^{-d})^{-d}=-f^{-d}$, while the minimal extension was recognized as $(-f)^{-d}$. In \cite{leg} the study continued to the case where the Lipschitz constant is allowed to vary. For a fixed $\alpha$, this is done by considering the coupling function $c:X\times(X\times\mathbb{R})\to\mathbb{R}$ defined by $c(x,(y,K))=-Kd^\alpha(x,y)$. Then, the $c$-convex functions on $X$ are characterized as the functions on $X$ which satisfy the $\alpha$-growth condition. This is also recalled later in \cite{sin}. The elementary $c$-convex functions in our discussion, that is, the functions $-Kd(\cdot,y)+r$ where $x\in X$ and $K,r\in\mathbb{R}$ are sometimes called \emph{cone functions}. Recently, attention has been drawn to operations, like the infimal convolution in $\mathbb{R}^n$, which are carried out with cone functions in order to solve the problem of \emph{absolutely minimizing Lipschitz extensions}; see, for example, \cite{ACJ}.

\section{Optimal Antiderivatives in the Representation of $c$-Monotone Mappings by $C$-Convex Functions}

In \cite{pen} the discussion regarding the representation of monotone mappings by convex functions is extended to the settings of $c$-convexity and $c$-monotonicity. In particular, the definition of the Fitzpatrick function and some of its basic properties are extended to hold for $c$-monotone mappings (see \eqref{mondef} in Definition \ref{cycmondef}). This discussion requires a type of the following convention: given $c: X \times Y \to \mathbb{R}$, we let
$$
C:(X\times Y)\times(Y\times X)\to\mathbb{R}
$$
be the function defined by
\begin{equation}
C((x,y),(t,s)):=c(x,t)+c(s,y),\ \ \ ((x,y),(t,s))\in (X\times Y)\times(Y\times X).
\end{equation}

\begin{definition}[Fitzpatrick function]
With a mapping $T:X\rightrightarrows Y$ (and the function $c$) we associate the Fitzpatrick
function $F_{[c,T]}:X\times Y\to\RX$, defined by
\begin{align}
F_{[c,T]}(x,y):&=\sup_{(s,t)\in G(T)}\ c(x,t)+c(s,y)-c(s,t)\\
&=\sup_{(s,t)\in G(T)}\ C((x,y),(t,s))-c(s,t),\ \ \ (x,y)\in X\times Y.
\end{align}

\end{definition}

When $T$ is proper and $c$-monotone, $F_{[c,T]}(x,y)\leq c(x,y)$ for each $(x,y)$
which is in $c$-monotone relations with the points of $G(T)$; in particular,
$F_{[c,T]}|_{G(T)}=c|_{G(T)}$. In this case $F_{[c,T]}$ is proper and hence also $C$-convex as a supremum of $C$-convex functions. When $T$ is maximal $c$-monotone,
$c\leq F_{[c,T]}$ with equality only at the points of $G(T)$. Another way to see that $F_{[c,T]}$ is $C$-convex is to observe that
\begin{equation}
F_{[c,T]}(x,y)=(c+\iota_{G(T)})^C(y,x)\ \ \ \ \ \ \forall(x,y)\in X\times Y.
\end{equation}
It follows, after permuting coordinates, that $F_{[c,T]}^C$ is the $C$-convexification of $c+\iota_{G(T)}$.
When $T$ is $c$-monotone, then for any $(x,y)\in X\times Y$ we have
\begin{align}
F_{[c,T]}^C(y,x)&=\sup_{(s,t)\in X\times Y}\ C((x,y),(t,s))-F_{[c,T]}(s,t)\geq\sup_{(s,t)\in G(T)}\ C((x,y),(t,s))-F_{[c,T]}(s,t)\\
&=\sup_{(s,t)\in G(T)}\ C((x,y),(t,s))-c(s,t)=F_{[c,T]}(x,y).
\end{align}
It now follows that when $T$ is maximal $c$-monotone, we have
\begin{equation}
c(x,y)\leq F_{[c,T]}(x,y)\leq F_{[c,T]}^C(y,x)\leq (c+\iota_{G(T)})(x,y)\ \ \ \forall (x,y)\in X\times Y,
\end{equation}
and each one of the inequalities becomes an equality for $(x,y)\in G(T)$.

The Fitzpatrick family of functions is defined by
\begin{equation}
\cF_{[c,T]}:=\big\{f\in\Gamma_C(X\times Y)\ |\ c\leq f,\ f|_{G(T)}=c|_{G(T)}\big\}.
\end{equation}
It is a well-known fact in the classical case that when $T$ is maximal monotone, then $F|_{[\langle\cdot,\cdot\rangle,T]}$ is the minimal function in $\cF|_{[\langle\cdot,\cdot\rangle,T]}$. We make a remark regarding this case towards the end of this section. For a general function $c$, we conclude (as in \cite{pen}) that when $T$ is maximal monotone, then $F_{[c,T]}\in\cF_{[c,T]}$ and, after permuting coordinates, $F_{[c,T]}^C\in\cF_{[c,T]}$. However, so far no minimality property of $F_{[c,T]}$ is available in the generality of $c$-monotone mappings. The framework of optimal antiderivatives does allow us to obtain a minimality property: $F_{[c,T]}$ is minimal in a certain family of $C$-convex $C$-antiderivatives which is determined by $T$. This is carried out below. In order to bring antiderivatives into the discussion we make use of the following notation.

\begin{definition}\label{D}
The subset $\{ ((x,y),(y,x))\ |\ (x,y)\in X\times Y\}$ of\ \ $(X\times
Y)\times(Y\times X)$ is denoted by $\D$. With a mapping
$T:X\rightrightarrows Y$, we associate the mapping $\D_T:X\times Y\rightrightarrows Y\times X$, defined by
\begin{equation}
G(\D_T)=\{ ((x,y),(y,x))\ |\ \ (x,y)\in T\}.
\end{equation}
\end{definition}

The following theorem is the main result of this section. It generalizes most of a similar result in the classical case from \cite{br}. The first part characterizes the $c$-monotonicity of $T$ in terms of $C$-cyclic monotonicity of $\D_T$ and in terms of $\D_T$ having a certain $C$-antiderivative. This gives rise to a family of $C$-antiderivatives of $\D_T$, of which, in the second part of the theorem, the Fitzpatrick function is the minimal member.

\begin{theorem}\label{main6}
(A) For a mapping $T:X\rightrightarrows Y$ the following assertions are equivalent:\\
1.\ $T$ is $c$-monotone;\\
2.\ $\D_T$ is $C$-monotone;\\
3.\ $\D_T$ is $C$-cyclically monotone;\\
4.\ $G(\D_T)\subset G(\partial_C(c+\iota_{G(T)}))$, that is, $c+\iota_{G(T)}$ is a $C$-antiderivative of $\D_T$.\\
\\
Consequently, the following assertions are equivalent:\\
1$'$.\ $T$ is maximal $c$-monotone;\\
2$'$.\ $\D_T$ is maximal $C$-monotone in $\D$;\\
3$'$.\ $\D_T$ is maximal $C$-cyclically monotone in $\D$;\\
4$'$.\ $T$ is a maximal subset in $X\times Y$ such that $G(\D_T)\subset G(\partial_C(c+\iota_{G(T)}))$.\\
\\
(B) If T is monotone, then
\begin{equation}
\alpha_{[C,c|_{G(T)},\D_T]}=F_{[c,T]}.
\end{equation}
If $T$ is maximal monotone, then
\begin{equation}\label{maxmon}
\cA_{[C,c|_{G(T)},\D_T]}\subset\cF_{[c,T]}.
\end{equation}
\end{theorem}

\begin{proof}
(A) For any $(x,y)$ and $(s,t)$ in $X\times Y$, we have
\begin{align}\label{cCmonotonicity}
&C((x,y),(y,x))-C((x,y),(t,s))-C((s,t),(y,x))+C((s,t),(t,s))\nonumber\\
=\ &2\big(c(x,y)-c(x,t)-c(s,y)+c(s,t)\big),
\end{align}
which implies $1\Leftrightarrow 2$. We also have
$$
c(x,y)+C((s,t),(y,x))=c(s,t)+C((x,y),(y,x))-[c(x,y)-c(x,t)-c(s,y)+c(s,t)],
$$
which implies that
$$
c(x,y)+C((s,t),(y,x))\leq c(s,t)+C((x,y),(y,x))\ \ \ \Leftrightarrow\ \ \ \ 0\leq c(x,y)-c(x,t)-c(s,y)+c(s,t).
$$
We conclude that $1\Leftrightarrow4$ holds. If $G(\D_T)\subset G(\partial_C(c+\iota_{G(T)}))$, then $\D_T$ is $C$-cyclically monotone. On the other hand, if $\D_T$ is $C$-cyclically monotone, then, in particular, $\D_T$ is $C$-monotone and therefore $T$ is $c$-monotone. Consequently, $G(\D_T)\subset G(\partial_C(c+\iota_{G(T)}))$. This implies that $3\Leftrightarrow4$ and completes the proof of (A).\\
\\
(B) Suppose that $T$ is $c$-monotone. In the setting of Theorem~\ref{main},
we let $M=\D_T$ and $f=c+\iota_{G(T)}$, so that $G(M)\subset G(\partial_C f)$ according to part (A). Letting $S=G(T)=\dom (\D_T)$, we may now consider the family $\cA_{[C,c|_{G(T)},\D_T]}$. Since $S=\dom (M)$, we may apply \eqref{conjufulldom} in order to compute the minimal $C$-antiderivative of $\D_T$ that agrees with $c$ on $G(T)$:
\begin{align}
 \alpha_{[C,c|_{G(T)},\D_T]}(x,y)
&=\sup_{((s,t),(t,s))\in G(\D_T)}\ [c(s,t)+C((x,y),(t,s))-C((s,t),(t,s))]\nonumber\\
\nonumber\\
&=\sup_{(s,t)\in G(T)}\ [c(x,t)+c(s,y)-c(s,t)]=\ F_{[c,T]}(x,y).\nonumber
\end{align}
If $T$ is maximal monotone and $h\in\cA_{[C,c|_{G(T)},\D_T]}$, then $c \leq F_{[c,T]}= \alpha_{[C,c|_{G(T)},\D_T]}\ \leq\ h$, which implies that $h\in\cF_{[c,T]}$.
\end{proof}\\

In the classical case, if $T$ is monotone, then
$$
\cA_{[C,c|_{G(T)},\D_T]}\supset\cF_{[c,T]},
$$
which, together with \eqref{maxmon}, implies that when $T$ is maximal monotone, then
$$
\cA_{[C,c|_{G(T)},\D_T]}=\cF_{[c,T]}.$$
Indeed, when $X$ is a topological vector space, $Y=X^*$ and $c=\langle\cdot,\cdot\rangle$, then we can rely on the fact that if $g$ is G\^{a}teaux differentiable at $x$ and $g\leq f$, where $f$ is convex and $g(x)=f(x)$, then $\nabla g(x)\in\partial f(x)$. In particular, for $f\in\cF_{[c,T]}$, since $\nabla c(x,y)=(y,x)$, then $f$ is a convex antiderivative of $\D_T$, hence, $f\in\cA_{[C,c|_{G(T)},\D_T]}$. However, even in the classical case, without the notion of minimal antiderivatives, when $T$ is monotone but not maximal monotone, there was a problem discussing the minimality of $F_{[c,T]}$. In this case, it may happen that $c\nleq F_{[c,T]}$, that is, $F_{[c,T]}$ is not a member of $\cF_{[c,T]}$. This gap has been filled in \cite{br}, where we consider the extended family $\cA_{[C,c|_{G(T)},\D_T]}$. There, for a monotone mapping $T$, we have seen that $\cA_{[C,c|_{G(T)},\D_T]}\supset\cF_{[c,T]}$ and that $F_{[c,T]}$ is the minimal function in $\cA_{[C,c|_{G(T)},\D_T]}$. A more detailed discussion in the classical case is available in \cite{br}.

In the classical setting in \cite{br}, the equivalence $1\Leftrightarrow 3$ also followed from a direct computation, without having an antiderivative at hand. This computation extends to $c$-monotone mappings without difficulty. It is a cyclic generalization of the identity \eqref{cCmonotonicity}. Indeed, let $\{(x_1,y_1),\dots,(x_n,y_n)\}\subset X\times Y$
and set $(x_1,y_1)=(x_{n+1},y_{n+1})$. We compute the cyclic sum of the ordered
pairs $((x_i,y_i)(y_i,x_i)),\ 1\leq i\leq n$:
\begin{align}
&\sum_{i=1}^n\ C((x_i,y_i),(y_i,x_i))-C((x_{i+1},y_{i+1}),(y_i,x_i))\nonumber\\
=&\sum_{i=1}^n\ c(x_i,y_i)-c(x_{i+1},y_i)-c(x_i,y_{i+1})+c(x_i,y_i)\nonumber\\
=&\sum_{i=1}^n\ c(x_i,y_i)-c(x_{i+1},y_i)-c(x_i,y_{i+1})+c(x_{i+1},y_{i+1}).\nonumber
\end{align}
Letting $n=2$, we get \eqref{cCmonotonicity}. We see that if the ordered pairs $(x_i,y_i),\ 1\leq i\leq n$, are in $c$-monotone relations, then the ordered pairs $((x_i,y_i)(y_i,x_i)),\ 1\leq i\leq n$, are in $C$-cyclically monotone relations. Conversely, if the latter are in $C$-cyclically monotone relations, then, in particular, they are in $C$-monotone relations. We now get the monotone relations of the ordered pairs $(x_i,y_i),\ 1\leq i\leq n$, from the case $n=2$.

\begin{example}
Let $(X,d)$ be a metric space, $X=Y$ and let $c=-d$. Let the metric $D:(X\times X)\times(X\times X)\to\RR$ be defined by $D=-C$ . The identity mapping $I_X$ is $-d$-monotone and we have
\begin{equation}\label{AIFI}
\cF_{[-d,I_X]}=\cA_{[-D,-d|_{G(I_X)},\D_{I_X}]}.
\end{equation}
The associated Fitzpatrick function and its $-D$-transform are given by
\begin{equation}\label{fitzident}
F_{[-d,I_X]}(x,y)=-d(x,y)\ \ \ \ \ and\ \ \ \ \ F_{[-d,I_X]}^{-D}(y,x)=d(y,x),\ \ \ \ \ \ \ (x,y)\in X\times X,
\end{equation}
respectively. If $T:X\rightrightarrows X$ is $-d$-monotone, then $G(T)\cup G(I_X)$ is a monotone set. Whenever either $T$ is maximal $-d$-monotone or $T$ is the $-d$-subdifferential of a $-d$-convex function, then
\begin{equation}\label{AFA}
\cA_{[-D,-d|_{G(T)},\D_T]}\subset\cF_{[-d,T]}\ \subset\ \cA_{[-D,-d|_{G(I_X)},\D_{I_X}]}.
\end{equation}
In particular, in this case,
\begin{equation}\label{FIFT}
F_{[-d,I_X]}(x,y)\leq F_{[-d,T]}(x,y)\leq F^{-D}_{[-d,T]}(y,x)=-F_{[-d,T]}(y,x)\leq F_{[-d,I_X]}^{-D}(y,x)\ \ \ \ \ \forall x,y\in X.
\end{equation}
\end{example}

\begin{proof}
The fact that $I_X$ is $-d$-monotone was discussed in Section 5. The formula for the Fitzpatrick function in \eqref{fitzident} follows from the triangle inequality and the fact that $d(x,x)=0$. Since $F_{[-d,I_X]}$ is $-D$-convex, its $-D$-transform is $-F_{[-d,I_X]}$ (according to Section 5). If $T$ is $-d$-monotone, then $\D_T$ is $-D$ cyclically monotone. Since $G(\D_T)\cup G(I_{X\times X})$ is a $-D$-cyclically monotone set, and since $G(\D_{I_X})\subset G(I_{X\times X})$, the union $G(\D_T)\cup G(\D_{I_X})$ is a $-D$-cyclically monotone set. It now follows that $G(T)\cup G(I_X)$ is a $-d$-monotone set. We conclude that whenever $T$ is either maximal $-d$-monotone or the $-d$-subdifferential of a $-d$-convex function, then $G(I_X)\subset G(T)$. In both cases, we have $-d=F_{[-d,I_X]}\leq F_{[-d,T]}$, and, after $-D$-transforming, we obtain \eqref{FIFT}. The fact that  $-d\leq F_{[-d,T]}$ also implies that $\cA_{[-D,-d|_{G(T)},\D_T]}\subset\cF_{[-d,T]}$. In order to complete the proof, it is now enough to prove that $\cF_{[-d,T]}\subset\cA_{[-D,-d|_{G(I_X)},\D_{I_X}]}$. Indeed, if $h:X\times X\to\RR$ is a $1$-Lipschitz function with respect to $D$ and $h(x,x)=0$ for all $x\in X$, then $-d=F_{[-d,I_X]}\leq h$. It follows that $h$ is a $-D$-antiderivative of $\D_{I_X}$, hence $h\in\cA_{[-D,-d|_{G(I_X)},\D_{I_X}]}$, as asserted.

\end{proof}

\section*{Acknowledgments}
This research was supported by the Israel Science Foundation (Grant
647/07), the Graduate School of the Technion, the Fund for the Promotion
of Research at the Technion and by the Technion's President's Research
Fund. We thank Constantin Z\u{a}linescu for sending us his fine and detailed comments regarding our discussion in \cite{br}, Eva Kopeck\'{a} for raising the issue of convexity of sets of antiderivatives, and the referee for a detailed and insightful report, with many helpful comments and suggestions.

\end{document}